\documentclass[11pt,reqno]{amsart}
\usepackage[left=3cm,top=2.5cm,right=3cm,bottom=2.5cm]{geometry}
\usepackage{amsmath,amsfonts}
\usepackage{amssymb,amsthm}
\usepackage{graphicx}
\usepackage[latin1]{inputenc}
\usepackage[T1]{fontenc}
\usepackage[english]{babel}

\theoremstyle{plain}
\newtheorem{Def}{Definition}
\newtheorem{Teo}{Theorem}

\newtheorem{Lema}[Teo]{Lemma}
\newtheorem{Prop}[Teo]{Proposition}

\theoremstyle{remark}
\newtheorem{Rem}{Remark}

\newcommand{\etal}{\textsl{et al.}}

\theoremstyle{remark}
\newtheorem{Example}{Example}

\begin{document}
\pagestyle{myheadings} 
\title{Unicyclic graphs with equal Laplacian energy}

\author[E. Fritscher]{Eliseu Fritscher}\email{\tt  eliseu.fritscher@ufrgs.br}

\author[C. Hoppen]{Carlos Hoppen}\email{\tt choppen@ufrgs.br}

\author[V. Trevisan]{Vilmar Trevisan} \email{\tt trevisan@mat.ufrgs.br}

\address{Instituto de Matem\'atica, UFRGS -- Avenida Bento Gon\c{c}alves, 9500, 91501--970 Porto Alegre, RS, Brazil}
\thanks{C. Hoppen acknowledges the support of FAPERGS (Proc.~11/1436-1)  and CNPq (Proc.~486108/2012-0 and~304510/2012-2). V. Trevisan was partially supported by CNPq (Proc. 309531/2009-8 and 481551/2012-3) and FAPERGS (Proc. 11/1619-2)}

\thanks{A paper with the same title has been accepted by Linear and Multilinear Algebra. Although the results are basically the same, the current manuscript contains a slightly modified version of Theorem~\ref{teodifespect}.}

\begin{abstract}
We introduce a new operation on a class of graphs with the property that the Laplacian eigenvalues of the input and output graphs are related. Based on this operation, we obtain a family of $\Theta(\sqrt{n})$ noncospectral unicyclic graphs on $n$ vertices with the same Laplacian energy.
\end{abstract}

\maketitle

\section{Introduction and main results}

In this paper, we deal with simple undirected graphs $G$ with vertex set $V=\{v_1,\ldots,v_n\}$. The \emph{Laplacian matrix} of $G$ is given by $L=D-A$, where $D$ is the diagonal matrix whose entry $(i,i)$ is equal to the degree of $v_i$ and $A$ is the adjacency matrix of $G$. The \emph{Laplacian spectrum} of $G$, denoted by $Lspect(G)$, is the (multi)set of eigenvalues of $L$, which will be written as $\mu_1 \geq \mu_2 \geq \cdots \geq \mu_n=0$. The \emph{Laplacian energy} of $G$, introduced by Gutman and Zhou~\cite{GZ06}, is given by
$$LE(G)=\sum_{i=1}^{n} |\mu_i-\overline{d}|,$$
where $\overline{d}$ is the average degree of $G$.

A natural question about the Laplacian energy concerns its power, as a spectral parameter, to discriminate graphs with the same number of vertices. In a sobering answer to this question, Stevanovi\'c \cite{Ste09} exhibited a set with $\Theta(n^2)$ threshold graphs on $n$ vertices having the same Laplacian energy. This large set of graphs with equal Laplacian energy seems to contrast with the case of trees. Stevanovi\'c  reports that, up to 20 vertices, there exists no pair of noncospectral trees with equal Laplacian energy. In fact, to the best of our knowledge, no pair of $n$-vertex noncospectral trees with the same Laplacian energy has been identified so far.

Finding a pair of $n$-vertex trees with equal Laplacian energy was the motivation of this work. Even though we have not succeeded, we did study a class of graphs that is \emph{close} to trees, namely the class of connected graphs with a single cycle, the so-called \emph{unicyclic graphs}. We asked  whether there exist $n$-vertex unicyclic graphs with equal Laplacian energy. The answer is affirmative. Indeed, we exhibit families with $\Theta(\sqrt{n})$ noncospectral unicyclic $n$-vertex graphs having the same Laplacian energy. To obtain these families, we introduce a graph operation that affects the Laplacian spectrum of a particular class of graphs in a way that can be controlled. This operation may lead to graph families that are relevant in other contexts and is interesting for its own sake.

To state our main results, we need to describe the structure of the graphs and of the operation under consideration.
\begin{Def}[Graph family $\mathcal{W}_{n,k}$]
Let $n,k$ be positive integers such that $n > 2k$. Consider a $k$-vertex graph $G^{\ast}$ whose vertices are labeled $1$ to $k$ and an $(n-2k)$-vertex graph $\breve{G}$ rooted at a vertex $u$. For any vector $y \in \{0,1\}^k$, we define an $n$-vertex graph $G=G(G^{\ast},\breve{G},y)$ by taking two disjoint copies of $G^*$ and one copy of $\breve{G}$, and by joining the root $u$ of $\breve{G}$ to the two copies of the vertex labeled $i$ in $G^*$ if and only if $y_i=1$. The graph family $\mathcal{W}_{n,k}$ comprises all graphs $G$ that can be
constructed in this way.
\end{Def}
We say that $G^{\ast}$ and $\breve{G}$ are the \emph{building blocks} of $G$, while $y$ is its \emph{adjacency vector}. Observe that some graphs $G\in\mathcal{W}_{n,k}$ may be constructed in more than one way.

Given a graph $G\in\mathcal{W}_{n,k}$, a \emph{canonical labeling} of the vertices of $G$ is given as follows. The original labeling of $G^*$ is used to label vertices in the two copies of $G^*$ in $G$ from $1$ to $k$ and from $k+1$
to $2k$, respectively. We let $v_{2k+1}=u$ and the remaining vertices of $\breve{G}$ are arbitrarily labeled $2k+2$ to $n$.

\begin{Example} Consider the labeled graph $G^*$
and the rooted graph $\breve{G}$ depicted in Figure~\ref{figexg}. If the
adjacency vector is given by $y=(1,1,0,0,0)^T\in\{0,1\}^5$, we obtain
a graph $G=G(G^{\ast},\breve{G},y)$ with 16 vertices.
\begin{figure}[h!]
       \begin{minipage}[c]{0.5 \linewidth}
           \fbox{\includegraphics[width=\linewidth]{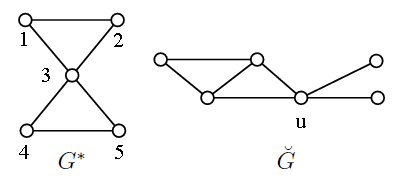}}\\
       \end{minipage}\hspace{2cm}
       \begin{minipage}[c]{0.25 \linewidth}
           \fbox{\includegraphics[width=\linewidth]{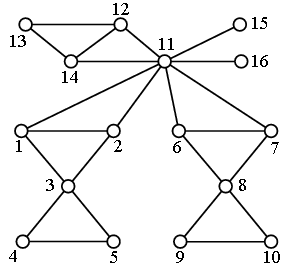}}\\
       \end{minipage}
       \caption{Graphs $G^*$, $\breve{G}$ and $G=G(G^*,\breve{G},y)$.}
       \label{figexg}
\end{figure} \end{Example}

We shall consider a specific operation that can be performed on graphs in $\mathcal{W}_{n,k}$.
\begin{Def}[Operation $\mathcal{E}_z$] Given a vector $z\in \{0,1\}^k$, the operation $\mathcal{E}_z$ is defined on a graph $G=G(G^{\ast},\breve{G},y)\in\mathcal{W}_{n,k}$ by inserting an edge between the two copies of the vertex labeled $i$ in $G^*$ if $z_i=1$. In other words,
$\mathcal{E}_z$ adds an edge between vertices $v_i$ and
$v_{k+i}$ of $G$ whenever $z_i=1$, for every $1\leq i\leq k$.
\end{Def}

We say that $z$ is the \textit{characteristic vector} of
$\mathcal{E}_z$.

\begin{Example} Consider the graph $G$ of Figure \ref{figexg}. Taking $z=(1,0,1,0,0)^T$ as
the characteristic vector, we obtain the graph
$\mathcal{E}_z(G)$ of Figure \ref{figexezg}.
\begin{figure}[h!]\begin{center}
    \fbox{\includegraphics[width=0.25\linewidth]{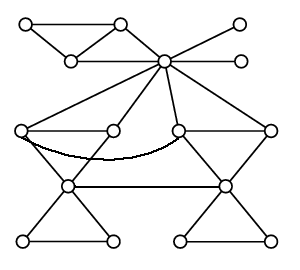}}
    \caption{Graph $\mathcal{E}_z(G)$}
    \label{figexezg}
\end{center}\end{figure}
\end{Example}

For a vector $y\in \{0,1\}^k$, we associate a
square matrix $E_y$ of order $k$ whose $i$-th column is the $i$-th
canonical vector $e_i\in\{0,1\}^k$ if $y_i=1$ and the null vector if $y_i=0$. So we can write $E_y=\sum_i y_i (e_i\cdot e_i^T)$.

The following result relates the Laplacian spectra of $G \in \mathcal{W}_{n,k}$ and $\mathcal{E}_z(G)$. Throughout the paper, the (multi)set of eigenvalues of a square matrix $A$ is denoted by $spect(A)$.
\begin{Teo}\label{teodifespect} Let $G$ be a graph in $\mathcal{W}_{n,k}$ with building blocks $G^{\ast}$ and $\breve{G}$, and adjacency vector $y$. Let $H=L(G^*)+E_y$. For $D=spect(H)$ and $F=spect(H+2E_z)$, where $z \in \{0,1\}^k$, we have
$$D \subset Lspect(G) \textsl{ and } Lspect(\mathcal{E}_z(G))=(Lspect(G) \setminus D)\cup F.$$
In particular, $G$ and $\mathcal{E}_z(G)$ have at least $n-k$ common Laplacian eigenvalues.
\end{Teo}

\begin{Example} Consider the graph $G \in \mathcal{W}_{11,3}$ in Figure~\ref{figexga} with building blocks $\breve{G}=\mathcal{C}_5$ and $G^*=P_3$, and adjacency vector $y=(1,1,1)^T$. If we choose $z=(1,1,1)^T$ as the characteristic vector, the matrices in the statement of Theorem~\ref{teodifespect} are given by 
\begin{eqnarray*}
L(G^*)=\begin{bmatrix}1 & -1 &0 \\ -1 & 2 & -1 \\0 & -1 & 1 \end{bmatrix}, & H=\begin{bmatrix}2 & -1 & 0 \\ -1 & 3 & -1 \\ 0 & -1 & 2 \end{bmatrix} and & H+2E_z=\begin{bmatrix}4 & -1 & 0 \\ -1 & 5 & -1 \\ & -1 & 4 \end{bmatrix},
\end{eqnarray*}
so that the sets $D$ and $F$ in the theorem satisfy $D = \{ 1, 2, 4 \}$ and $F=\{3,4,6\}$. In particular, we have
\begin{eqnarray*}
Lspect(G)&=&\{0, 0.49257, 1, 1.38197, 2, 2, 2.47142, 3.61803, 4, 4, 9.03601\},\\
Lspect(\mathcal{E}_z(G))&=&\{0, 0.49257, 1.38197, 2, 2.47142, 3, 3.61803, 4,
4, 6, 9.03601\}.
\end{eqnarray*}
\begin{figure}[h!]
       \begin{minipage}[c]{0.3 \linewidth}
           \fbox{\includegraphics[width=\linewidth]{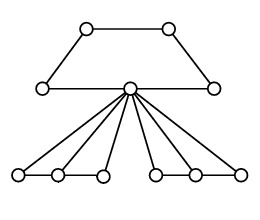}}\\
       \end{minipage}\hspace{2cm}
       \begin{minipage}[c]{0.3 \linewidth}
           \fbox{\includegraphics[width=\linewidth]{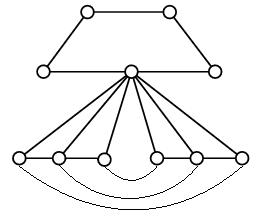}}\\
       \end{minipage}
       \caption{Graphs $G$ and $\mathcal{E}_z(G)$}
       \label{figexga}
\end{figure}
\end{Example}

We shall concentrate on a special class of graphs in $\mathcal{W}_{n,k}$. Recall that a tree is \emph{starlike} if it has a unique vertex with degree larger than two (the degree is therefore equal to the number of leaves in the tree). We focus on a particular class of starlike trees.
\begin{Def}[Graph family $\mathcal{S}_{n,k}$] A graph $G$ lies in  $\mathcal{S}_{n,k}$ if it is a
starlike tree whose central vertex $u$ is adjacent to one of the ends of $h\geq3$
paths $P_{a_i}$, where $a_i$ is even for $1\leq i\leq h-1$, $a_1=a_2=k\geq2$
and $a_h < n/2$ is odd.
\end{Def}
The paths $P_{a_i}$ are called the \emph{branches} of the starlike tree $G \in \mathcal{S}_{n,k}$. In particular, the single path $P_{a_h}$ with an odd number of vertices is the \emph{odd branch} of $G$.

Clearly, given a graph $G\in\mathcal{S}_{n,k}$, it may be viewed as a graph in
$\mathcal{W}_{n,k}$: its building blocks are $G^*=P_k$, whose vertices are labeled in increasing order along the path, and $\breve{G}$, which is rooted at the central vertex of the starlike tree and is obtained from $G$ by removing two occurrences of $P_k$. The adjacency vector is $y=e_k=(0,\ldots,0,1)^T$. Observe that the same tree may belong to $\mathcal{S}_{n,k}$
for different values of $k$.  For instance, Figure~\ref{figexsk} depicts a tree
that is in both $\mathcal{S}_{16,2}$ and $\mathcal{S}_{16,4}$.\\
\begin{figure}[h!]
    \begin{center}\includegraphics[width=0.5\linewidth]{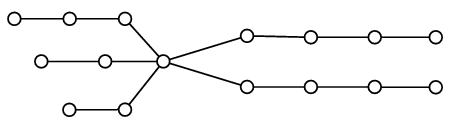}\end{center}
    \caption{A graph in $\mathcal{S}_{16,2}$ and in $\mathcal{S}_{16,4}$}
    \label{figexsk}
\end{figure}

Our interest in this family is justified by the fact that, given a graph $G \in  \mathcal{S}_{n,k}$, we are able to determine precisely which are the $k$ eigenvalues in the set $D$ defined in Theorem~\ref{teodifespect}, and which are the $k$ values that replace them in the Laplacian spectrum of $\mathcal{E}_{e_1}(G)$, where $e_1=(1,0,\ldots,0)^T \in \{0,1\}^k$. Furthermore, and crucially, we are able to prove the following.

\begin{Teo}\label{teoenlapl} Every $G\in\mathcal{S}_{n,k}$ satisfies $LE(G)=LE(\mathcal{E}_{e_1}(G)).$
\end{Teo}

As a direct consequence of Theorem~\ref{teoenlapl}, we derive the following result, which we deem to be the main result in this paper.
\begin{Teo} \label{cor} For every $\ell \geq 2$, there is a family of $\ell$ noncospectral unicyclic graphs with the same Laplacian energy, each with $n=2\ell^2+2\ell+2$ vertices. In particular, for values of $n$ of this type, there is a family of  $\Theta(\sqrt{n})$ noncospectral unicyclic graphs on $n$ vertices with the same Laplacian energy.
\end{Teo}
The problem of generating families of noncospectral equienergetic graphs has attracted a good deal of attention in the context of the (standard) energy associated with a graph, which was introduced by Gutman~\cite{G78} and is based on the spectrum of the adjacency matrix. To cite one of the many developments in this direction, we mention the work of Ramane~\etal~\cite{Ram2004}, who showed that there are infinitely many pairs of noncospectral equienergetic graphs so that the graphs in each pair are connected and have the same number of vertices and edges. Our families of unicyclic graphs with the same Laplacian energy may be seen as a counterpart of this result. Moreover, Li and So~\cite{LS} constructed infinitely many pairs of equienergetic graphs where one of the graphs is obtained from the other by deleting an edge. We have found pairs with the same property in the Laplacian context, namely the pairs $(\mathcal{E}_{e_1}(G),G)$ with $G \in \mathcal{S}_{n,k}$.

The remainder of the paper is organized as follows. In Section~\ref{def}, we study the way in which the Laplacian spectrum of the elements of $\mathcal{W}_{n,k}$ is affected by the operation $\mathcal{E}_z$. This characterization leads to the proofs of Theorem~\ref{teoenlapl} and Theorem~\ref{cor} in Section~\ref{energy}. Section~\ref{additional} contains the proofs of a few technical results used in the previous sections.

\section{The connection between $Lspect(G)$ and $Lspect(\mathcal{E}_z(G))$}\label{def}

In this section, we prove Theorem~\ref{teodifespect}, which relates the Laplacian spectrum of a
graph $G\in\mathcal{W}_{n,k}$ with the Laplacian spectrum of $\mathcal{E}_z(G)$.

\begin{proof}[Proof of Theorem~\ref{teodifespect}] Let $G$ be a graph in $\mathcal{W}_{n,k}$ with building blocks $G^{\ast}$ and $\breve{G}$, and adjacency vector $y$. Assume that the vertex set of $G$ is ordered according to a canonical labeling. The Laplacian matrix $L=L(G)$ has the form
\begin{equation}\label{tapsi1}
L=\begin{bmatrix} H &&-y& \\ &H&-y& \\ -y^T&-y^T&\delta&t^T \\ &&t&B
\end{bmatrix}
\end{equation}
where $H=L(G^*)+E_y$ is the matrix of order $k$ that coincides with the Laplacian
matrix of $G^*$, except for the diagonal, where each entry is assigned one unit more
if the respective vertex is adjacent to $u$. Moreover, $\delta=d(u)$
is the degree of the vertex $u$, while $B$ and $t$ are, respectively, a submatrix of $L(\breve{G})$ and a vector, both of order $n-k-1$, associated with the remaining $n-2k-1$ vertices of $G$ (that is, with the vertices of $\breve{G}-u$).

We first show that $D \subset Lspec(G)$. Let $\alpha_1,\ldots,\alpha_k$ be the eigenvalues of $H$ (listed according to their multiciplity). Since $H$ is symmetric, we may associate an eigenvector $v_i$ with each $\alpha_i$ so that $\{v_1,\ldots,v_k\}$ is an orthogonal basis of $\mathbb{R}^k$.

In the remainder of this proof, a vector $w \in \mathbb{R}^n$ will be written as $w=(a^T,b^T,c,d^T)^T$, where
$a,b\in\mathbb{R}^k$, $c\in\mathbb{R}$ and $d\in\mathbb{R}^{n-2k-1}$.
For $w_j=(v_j^T,-v_j^T,0,0^T)^T$, we have
\begin{eqnarray*}
L\cdot w_j=\begin{bmatrix} H &&-y& \\ &H&-y& \\ -y^T&-y^T&\delta&t^T \\
&&t&B \end{bmatrix} \begin{bmatrix}v_j\\-v_j\\0\\0\end{bmatrix}=
\begin{bmatrix}H\cdot v_j -0y\\-H\cdot v_j-0y\\-y^Tv_j+y^Tv_j+0\delta+t^T\cdot 0\\
0t+B\cdot 0\end{bmatrix}= \begin{bmatrix}\alpha_j v\\ -\alpha_j v\\0\\0
\end{bmatrix}=\alpha_j w_j,
\end{eqnarray*}
implying that $D \subset Lspect(G)$.

We now prove that, for any characteristic vector $z \in \{0,1\}^k$, the remaining $n-k$ eigenvalues of $L(G)$ are Laplacian eigenvalues of $\mathcal{E}_z(G)$. Since $\{v_1,\ldots,v_k\}$ is a basis of the space $\mathbb{R}^k$, there are constants $\beta_{i,j}$, for $i,j\in\{1,\ldots,k\}$, such that $$\sum_{i=1}^k \beta_{i,j} v_i = e_j,$$
where $e_j$ is the $j$-th canonical vector in $\mathbb{R}^k$. Clearly, we also have $$\sum_{i=1}^k \beta_{i,j} w_i = \begin{bmatrix}e_j\\-e_j\\0\\0
\end{bmatrix}:=e_j^*.$$
Since the matrix $L$ is symmetric, we may turn the set $\{e_1^*,\ldots,e_k^*\}$ into a basis of $\mathbb{R}^n$ by adding $n-k$ orthogonal eigenvectors of $L$, which are also orthogonal to all $w_i$ and, consequently, orthogonal to all $e_j^*$. Let $\lambda \in Lspect(G)\setminus D$ with eigenvector $w$. The
Laplacian matrix of $\mathcal{E}_z(G)$ has the form
\begin{equation}\label{L(E)}
L(\mathcal{E}_z(G))=L(G)+E=\begin{bmatrix} H &&-y& \\ &H&-y& \\
-y^T&-y^T&\delta&t^T\\ &&t&B \end{bmatrix}+\begin{bmatrix}E_z&-E_z&&\\
-E_z&E_z&&\\&&0&\\&&&0\end{bmatrix}.
\end{equation}
Note that each of the first $k$ rows of $E$ is either ${e_j^*}^T$ (for some $j$) or a row of zeros, and that each of the next $k$ rows is either $-{e_j^*}^T$ (for some $j$) or a
row of zeros, so that
\begin{equation*}
L_{\mathcal{E}_z(G)}\cdot w = L_{G}\cdot w + E\cdot w = \lambda w +
\begin{bmatrix}0\\ \vdots\\{e_j^*}^T\cdot w\\ \vdots \end{bmatrix} =
\lambda w,
\end{equation*}
because $w$ is orthogonal to every $e_j^*$. 

To conclude the proof, we find $k$ eigenvalues of $\mathcal{E}_z(G)$ whose corresponding eigenvectors generate the vector space spanned by $\{e_1^*,\ldots,e_k^*\}$. To this end, let $\{\gamma_1,\ldots,\gamma_k\}$ be the (multi)set of eigenvalues of $H+2E_z$ and let $\{v_1,\ldots,v_k\}$ be an orthogonal set of eigenvectors such that each vector $v_i$ corresponds to the eigenvalue $\gamma_i$. Setting $w_i=(v_i^T,-v_i^T,0,0^T)^T$, we have
\begin{equation*}
L_{\mathcal{E}_z(G)}\cdot w_i = L_{G}\cdot w_i + E\cdot w_i = \begin{bmatrix}H\cdot v_i\\ -H\cdot v_i\\ 0 \\0 \end{bmatrix} +
\begin{bmatrix}2E_z\cdot v_i\\ -2E_z\cdot v_i \\ 0 \\0 \end{bmatrix} = \begin{bmatrix}\gamma v_i\\ -\gamma v_i \\ 0 \\0 \end{bmatrix}=\gamma w_i.
\end{equation*}
So each $\gamma_i\in F=spect(H+2E_z)$ is an eigenvalue of $L_{\mathcal{E}_z(G)}$, and the set $\{w_1,\ldots,w_k\}$ spans the vector space with basis $\{e_1^*,\ldots,e_k^*\}$, as required.
\end{proof}

Our next objective is to study the Laplacian spectrum of a graph $G \in \mathcal{S}_{n,k}$. More precisely, in the case when the characteristic vector $z$ is given by $e_1 \in \{0,1\}^k$, we determine the sets $D$ and $F$ associated with a graph $G  \in \mathcal{S}_{n,k}$, which are defined in Theorem~\ref{teodifespect}. Actually, we prove this result for a slightly more general class of graphs, which we call $\mathcal{S}_{n,k}^\ast$ and which contains all graphs in $\mathcal{W}_{n,k}$ such that $G^{\ast}$ is a path $P_k$ (not necessarily even) and $y$ is the canonical vector $e_k$. (Observe that $\breve{G}$ is arbitrary.)

To state our result precisely, given a positive integer $k$, let
$$D_k=\left\{2+2\cos{\frac{2j\pi}{2k+1}}:j=1,\ldots,k\right\} \text{~and }F_k=\left\{2-2\cos{\frac{2j\pi}{2k+1}}:j=1,\ldots,k\right\}.$$
\begin{Prop}\label{propdifespectSk} If $G\in\mathcal{S}_{n,k}^{\ast}$ and $z \in \{0,1\}^k$, then
$Lspect(G)\setminus D_k \subset Lspect(\mathcal{E}_{z}(G))$. Moreover, for $z=e_1$, we have
$Lspect(\mathcal{E}_{e_1}(G))=(Lspect(G)\setminus D_k)\cup F_k$.
\end{Prop}

To prove Proposition~\ref{propdifespectSk}, we shall compute the sets $D$ and $F$ of Theorem~\ref{teodifespect}. To this end, the following technical lemma will be particularly useful. For a proof of this result, see Yueh~\cite[Theorem 1 and 2]{yueh}.
\begin{Lema}\label{lemaAs} Let $A_s$ be a tridiagonal matrix such that
\begin{equation}\label{eqAS}
A_s=\begin{pmatrix}
-\alpha+b & c &  &  \\
a & b & \ddots &  \\
& \ddots & b & c \\
&  & a & -\beta+b \\ \end{pmatrix} \in \mathbb{R}^{s \times s}.
\end{equation}
If $|\alpha|=\sqrt{ac}\neq0$ and $\beta=0$, then $Spect(A_s)=\left\{b+2\alpha\cos{\frac{2j\pi}{2s+1}}:j=1,\ldots,s\right\}$.
\end{Lema}

\begin{proof}[Proof of Proposition~\ref{propdifespectSk}] By Theorem~\ref{teodifespect}, for every $z \in \{0,1\}^k$, we have the relation
$Lspect(G)\setminus D \subset Lspect(\mathcal{E}_{z}(G))$, where
$D=spect(L(P_k)+E_{e_k})$. Since
$L(P_k)+E_{e_k}=A_k$, we have $D=D_k$ by Lemma \ref{lemaAs}, where $A_k$ is defined in~\eqref{eqAS} for $a=c=-1, b=2, \beta=0$ and $\alpha=\sqrt{ac}=1$. On the other hand, for $z=e_1$, we have $F=spect(H+2E_{e_1})$ and we obtain $F=F_k$ because $H+2E_{e_1}=A_k$ in~\eqref{eqAS}, where $a=c=-1, b=2, \beta=0$ and $\alpha=-1$.
\end{proof}

\section{Families of Laplacian equienergetic unicyclic graphs}\label{energy}

We use the results of the previous section to find families of noncospectral unicyclic graphs with the same Laplacian energy. Observe that, using the identity $\sum_{i=1}^{n} \mu_i = n \overline{d}$, we may express the Laplacian energy of a graph $G$ as $$LE(G)=2\sum_{i=1}^\sigma \mu_i - 2\sigma\overline{d},$$ where $\sigma$ is the number of eigenvalues larger than or equal to the average degree $\overline{d}$ of $G$. Our objective here is to provide a proof of Theorem~\ref{teoenlapl}, that is, we wish to show that, for every $G\in \mathcal{S}_{n,k}$, we have $LE(G)-LE(\mathcal{E}_{e_1}(G))=0.$

For a graph $G$, we let $\mu_i^G$, $\overline{d}^G$ and $\sigma^G$ be, respectively, the $i$-th largest Laplacian eigenvalue of $G$, the average degree of $G$ and the number of eigenvalues that are larger than or equal to the average degree of $G$.

\begin{Lema}{\label{lemaDeltaLE}} Let $G$ and $G'$ be $n$-vertex graphs such that $\sigma^G=\sigma^{G'}=\sigma$. We have $$\Delta LE(G',G)=LE(G')-LE(G) = 2\sum_{i=1}^\sigma (\mu_i^{G'}-\mu_i^G) - \frac{4\sigma\Delta e}{n},$$
where $\Delta e=e(G')-e(G)$.
\end{Lema}

\begin{proof} Since $\overline{d}^G=\frac{2e(G)}{n}$, the expression for $\Delta LE$ is
\begin{eqnarray*}
\Delta LE &=& 2\sum_{i=1}^\sigma \mu_i^{G'} - 2\sigma\overline{d}^{G'} - 2\sum_{i=1}^\sigma \mu_i^G + 2\sigma\overline{d}^G \\
&=& 2\sum_{i=1}^\sigma (\mu_i^{G'}-\mu_i^G) - 2\sigma\left(\frac{2e(G')-2e(G)}{n}\right)\\
&=& 2\sum_{i=1}^\sigma (\mu_i^{G'}-\mu_i^G) - 2\sigma\left(\frac{2\Delta e}{n}\right).
\end{eqnarray*}
\end{proof}

In light of this result, it will be convenient to know the number of Laplacian eigenvalues of a graph that are larger than or equal to their average value. This is settled by the following lemma, which will be proved in the next section.

\begin{Lema}\label{lemasigmag} Every graph $G\in\mathcal{S}_{n,k}$ satisfies  $\sigma^G=\sigma^{\mathcal{E}_{e_1}(G)}=\frac{n}{2}$.
\end{Lema}

\begin{proof}[Proof of Theorem~\ref{teoenlapl}] Let $G \in \mathcal{S}_{n,k}$. Because of Lemma~\ref{lemasigmag}, we may apply Lemma~\ref{lemaDeltaLE} to $G$ and $G'=\mathcal{E}_{e_1}(G)$ to obtain
$$\Delta LE = \Delta LE(G',G) = 2\sum_{i=1}^\sigma (\mu_i^{G'}-\mu_i^G) - \frac{4\sigma\Delta e}{n} = 2\sum_{i=1}^{\frac{n}{2}}(\mu_i^{G'}-\mu_i^G) - 2,$$
since $\Delta e=1$ and $\sigma= n/2$. We now verify which eigenvalues of $G$ and $G'$ are above or below average (where $\overline{d}^G=2-2/n$ and $\overline{d}^{G'}=2$).

Note that $\cos(2j\pi/(2k+1))$ is a decreasing function of $j \in \{1,2,\ldots,k\}$, which is nonnegative if and only if $j \leq \lfloor k/2+1/4 \rfloor=k/2$. Moreover, for $j=\frac{k}{2}+1$, we have
\begin{eqnarray*}
-\cos{\left(\frac{(k+2)\pi}{2k+1}\right)}&=&\sin{\left(\frac{3\pi}{4k+2}\right)}>\frac{3\pi}{4k+2}-\frac{1}{3!}\left(\frac{3\pi}{4k+2}\right)^3>\frac{3\pi}{4k+4}-\frac{1}{6}\left(\frac{3\pi}{4k}\right)^3\\
&=& \frac{3\pi-2}{4k+4}-\frac{9\pi^3}{128k^3}+\frac{1}{2k+2}>\frac{237k^3-280k-280}{128k^3(k+1)}+\frac{1}{2k+2}>\frac{1}{2k+2},
\end{eqnarray*}
for $k\geq2$. Hence
\begin{eqnarray*}
2+2\cos{\left(\frac{(k+2)\pi}{2k+1}\right)}<2-\frac{2}{2k+2}\leq 2-\frac{2}{n}.
\end{eqnarray*}
We conclude that $\alpha_{\frac{k}{2}+1}<\overline{d}^G$, implying that exactly $\frac{k}{2}$ elements of $D_k$ are larger than or equal to the average $\overline{d}^G$, namely $\alpha_j=2+2\cos{\left(\frac{2j\pi}{2k+1}\right)}$, $1\leq j \leq k/2$.

On the other hand, an element $2-2\cos{\left(\frac{2j \pi}{2k+1}\right)} \in F_k$ is larger than or equal to $\overline{d}^{G'}=2$ if and only if $j\geq\frac{k}{2}+1$. This means that there are exactly $\frac{k}{2}$ elements of $F_k$  that are larger than or equal to $\overline{d}^{G'}$. The remaining $\sigma-\frac{k}{2}$ largest eigenvalues of $G$ and $G'$ coincide by Proposition~\ref{propdifespectSk}. Let $S$ denote this set of common eigenvalues. We have
\begin{eqnarray*}
\Delta LE &=& 2\sum_{i=1}^\frac{n}{2} (\mu_i^{G'}-\mu_i^G) - 2\\
&=& 2\sum_{\mu_i^{G'},\mu_j^G \in S}(\mu_i^{G'}-\mu_j^G)+ 2\sum_{j=1}^{\frac{k}{2}}\left[\left(2-2\cos{\frac{2(j+\frac{k}{2})\pi}{2k+1}}\right)-\left(2+2\cos{\frac{2j\pi}{2k+1}}\right)\right]-2\\
&=&-4\sum_{j=1}^{\frac{k}{2}}\left(\cos{\frac{2(j+\frac{k}{2})\pi}{2k+1}}+\cos{\frac{2j\pi}{2k+1}}\right)-2
\end{eqnarray*}
To conclude the proof, we use the trigonometric identity
$$\sum_{j=1}^{k} \cos{\frac{2j\pi}{2k+1}} = -\frac{1}{2},$$
which leads to $\Delta LE =0$, as required. For completeness, we give a proof of this inequality at the end of the paper (see Lemma~\ref{lemasumcos}).
\end{proof}

\begin{Rem}
Our definition of $\mathcal{S}_{n,k}$ has the restriction $a_h<n/2$ on the length of the odd path because we need this hypothesis in our proof of Lemma~\ref{lemasigmag}. Numerical experiments suggest that this lemma should hold without this restriction.
\end{Rem}

We are now ready to prove Theorem~\ref{cor}, the main theorem of this paper. In fact, we prove the following more general version of this result. (Theorem~\ref{cor} is just Theorem~\ref{cor2} with $\gamma=1$.)
\begin{Teo} \label{cor2} Let $\ell\geq 2$ and $\gamma\geq1$ be integers. There exists a family of $\ell$ unicyclic noncospectral graphs on $n=2\ell^2+2\ell+2\gamma$ vertices with the same Laplacian energy. Moreover, for each $\gamma\geq2$ there are at least two such families.
\end{Teo}

\begin{proof} Consider a graph $G\in\mathcal{S}_{n,2} \cap \mathcal{S}_{n,4} \cap \cdots \cap \mathcal{S}_{n,2\ell}$ given by a central vertex $u$ adjacent to two copies of $P_{2i}$ for every $1\leq i \leq \ell$ and to one copy of $P_1$. The graph constructed so far has $2+2\sum_{i=1}^\ell 2i=2(\ell^2+\ell+1)$ vertices. We distribute the remaining $2(\gamma-1)$ vertices in pairs, either adding them to new paths with an even number of vertices adjacent to $u$ or increasing the branch of odd length, making sure that it does not reach length $\frac{n}{2}$.

Considering $G$ as a graph in $\mathcal{S}_{n,2i}$, for $1\leq i\leq \ell$, we build unicyclic graphs $G_i=\mathcal{E}_{e^{(2i)}_1}(G)$, where $e_1^{(2i)}$ is the canonical vector $e_1$ viewed as a vector in $\mathbb{R}^{2i}$. In particular, each graph $G_i$ contains the cycle $\mathcal{C}_{2i+1}$. The $\ell$ graphs $G_i$ constructed in this way have the same Laplacian energy as $G$ by Theorem \ref{teoenlapl}. Moreover, no pair of graphs in this family is
cospectral. Indeed, the multiplicity of the eigenvalue $\alpha_i=2+2\cos\left(\frac{2\pi}{4i+1}\right)$ in $G_i$ is smaller than in $G$ (since this eigenvalue lies in $D_i$, but not in $F_i$); however, $\alpha_i\notin D_j$ for $i \neq j$ so that the spectra of $G_i$ and $G_j$ differ.

Clearly, for each configuration of the $2(\gamma-1)$ additional vertices of $G$, we create a different family of graphs with the same Laplacian energy.
\end{proof}

\begin{Rem}
Based on Theorem~\ref{cor2}, we may easily extend the conclusion of Theorem~\ref{cor} to all even values of $n$. In other words, for all even values of $n$, there is a family of  $\Theta(\sqrt{n})$ noncospectral unicyclic graphs on $n$ vertices with the same Laplacian energy. Indeed, let $\ell$ be the largest integer such that $2\ell^2+2\ell< n$ and, for these values of $\ell$ and $n$, construct a family of graphs with $\gamma=(n-2\ell^2-2\ell)/2$ as in Theorem~\ref{cor2}. It contains $\ell$ noncospectral unicyclic graphs on $n$ vertices with the same Laplacian energy.
\end{Rem}

\begin{Example} To conclude this section, we use the construction in the proof of Theorem~\ref{cor2} to obtain $\ell=4$ noncospectral unicyclic graphs with the same Laplacian energy. We are able to build graphs with $n=40+2\gamma$ vertices, for any integer $\gamma \geq 1$. For $\gamma=2$, we may insert the $2(\gamma-1)=2$ additional vertices as an extra path $P_2$ adjacent to $u$. The family is depicted in Figure \ref{figexeqlapl} and their Laplacian energy is approximated by $LE=60.70698$.
\begin{figure}[h!]
       \begin{minipage}[c]{0.19 \linewidth}
           \fbox{\includegraphics[width=\linewidth]{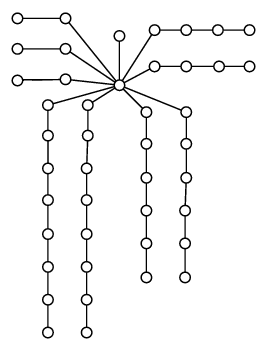}}\\
       \end{minipage}
       \begin{minipage}[c]{0.19 \linewidth}
           \fbox{\includegraphics[width=\linewidth]{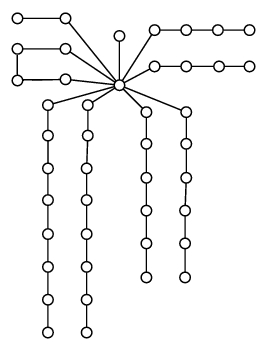}}\\
       \end{minipage}
       \begin{minipage}[c]{0.19 \linewidth}
           \fbox{\includegraphics[width=\linewidth]{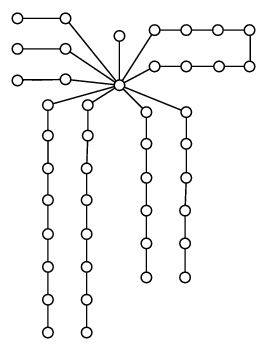}}\\
       \end{minipage}
       \begin{minipage}[c]{0.19 \linewidth}
           \fbox{\includegraphics[width=\linewidth]{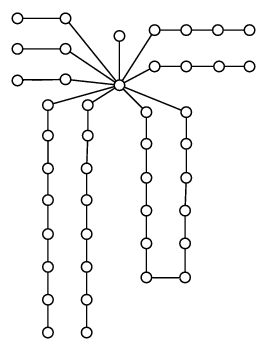}}\\
       \end{minipage}
       \begin{minipage}[c]{0.19 \linewidth}
           \fbox{\includegraphics[width=\linewidth]{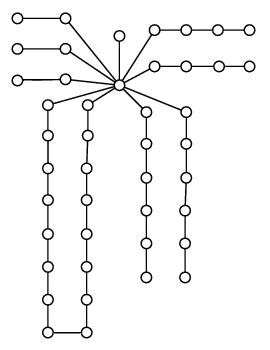}}\\
       \end{minipage}
       \caption{$G \in \bigcap_{i=1}^4 \mathcal{S}_{44,2i}$ and the four equienergetic graphs obtained from it.}
       \label{figexeqlapl}
\end{figure}
\end{Example}

\section{Additional proofs}\label{additional}

In this section, we establish two technical results that were useful in our proofs. In Section~\ref{energy} it was necessary to compute, for a graph $G \in \mathcal{S}_{n,k}$, the number $\sigma^G$ of Laplacian eigenvalues that are larger than or equal to the average degree of $G$. Indeed, we relied on Lemma~\ref{lemasigmag}, which states that $G$ and $\mathcal{E}_{e_1}(G)$ have the same number of such eigenvalues, namely $n/2$. We shall now prove this result.

To this end, we use an algorithm due to Jacobs and Trevisan \cite{JT11}, which was originally stated in terms of the adjacency matrix, but may be readily adapted to the Laplacian matrix (see~\cite{FHRT11} for details). It enables us to determine the number of
(Laplacian) eigenvalues of a tree that are larger than $\alpha$, equal to
$\alpha$ and smaller than $\alpha$, where $\alpha$ is an arbitrary real number.
For the sake of completeness we give a brief description of the algorithm.
An arbitrary vertex is chosen as the root of an $n$-vertex input tree, and the vertices are labeled $1$ to $n$, bottom up with respect to the root (i.e., each vertex has a higher label than its children). The algorithm is initialized by assigning the value $a(v_i)=d(v_i)-\alpha$ to each vertex $v_i$, where $d(v_i)$ is the degree of $v_i$. Then the
vertices are processed one by one, according to the order given by the labeling: leaves are left
unchanged, while, for each (nonleaf) vertex, the algorithm assigns a new value $a(v_i)\leftarrow
a(v_i) - \sum_{v\in C_i}\frac{1}{a(v)}$, where $C_i$ is the set of children of
$v_i$, provided that $0\notin\{a(v):v\in C_i\}$. If $0\in\{a(v):v\in C_i\}$ the algorithm chooses a vertex $v_k$ in $C_i$ such that
$a(v_k)=0$ and performs the following steps: $a(v_i)\leftarrow-\frac{1}{2}$, $a(v_k)\leftarrow2$ and, if $v_i$ is not the root,
the algorithm suppresses the edge between $v_i$ and its parent. At the end of the process, the
number of occurrences of positive, negative and zero values $a(v)$ corresponds to the number of Laplacian eigenvalues that are larger than, smaller than and equal to $\alpha$, respectively.
\begin{proof}[Proof of Lemma~\ref{lemasigmag}]
Let $G\in\mathcal{S}_{n,k}$. We wish to show that $\sigma^{G}=\sigma^{\mathcal{E}_{e_1}(G)}=\frac{n}{2}$.

We apply the above algorithm on $G$, rooted at $u$, with $\alpha=2$.
In the beginning, the leaves are assigned $a(v)=-1$, the central vertex $u$, adjacent to
$r \geq 3$ branches, is assigned $a(u)=r-2$ and all the other vertices are assigned $a(v)=0$.
The algorithm processes the vertices one by one, from the leaves towards $u$, so that, for all branches, we obtain the values $a(v)=-1$ for vertices in odd positions, $a(v)=1$ for vertices in even positions, and $a(u)=r-2-(r-1)+1=0$.
Therefore, the number of Laplacian eigenvalues larger than or equal to $\alpha=2$
in $G$ is $\frac{n}{2}$.  It follows from the proof of Theorem~\ref{teoenlapl} that $k/2$ of these eigenvalues lie in $D_k$, and that $k/2$ eigenvalues in $F_k$ are larger than or equal to 2. Since $Lspect(G')=(Lspect(G)\setminus D_k)\cup F_k$, we have $\sigma^{G'}=\frac{n}{2}$.

To determine $\sigma^{G}$, we apply the same algorithm to $G$ with $\alpha=\overline{d}^{G}=2-\frac{2}{n}$. Upon initialization, each leaf is assigned $a(v)=-1+\frac{2}{n}$, the vertices of degree two receive
$a(v)=\frac{2}{n}$, and $a(u)=r-2+\frac{2}{n}$. We consider the performance of the algorithm on each branch $P_{a_i}$ of $G$. More precisely, we shall prove that the number of positive entries on $P_{a_i}$ at the end of the algorithm is at most $\lfloor a_i/2 \rfloor$. Observe that this leads to our result: indeed, it implies that the number of positive entries over all vertices of $G$ other than $u$ is at most $(n-2)/2$ (recall that $n-1$ is odd). The root $u$ may still contribute with an additional positive entry, which leads to $\sigma^G\leq\frac{n}{2}$. On the other hand, we already know that $\frac{n}{2}$ Laplacian eigenvalues of $G$ are larger than or equal to 2, so that $\sigma^G=\frac{n}{2}$.

We now prove our claim. As an auxiliary result, we use the fact that there are precisely $\lfloor\frac{t}{2}\rfloor$ Laplacian eigenvalues of a path $P_t$ that are larger than or equal to the average $\overline{d}^{P_t}=2-2/t$. This is well known and may be derived directly from the Laplacian spectrum of a path (which may be found in~\cite{Brobook}, for instance).

First assume that $a_i$ is even. Consider an application of the algorithm to the graph $G^{\star}=P_{a_i+1}$ (rooted at one of the leaves) with $\alpha^{\star}=2-2/(a_i+1)$. The auxiliary result tells us that exactly $a_i/2$ entries will be nonnegative. Observe that, if the algorithm were applied to $G$ with the same value $\alpha^{\star}$, the outcome would be exactly the same on the branch $P_{a_i}$. Since $a_i+1<n$ (and hence $\alpha>\alpha^{\star}$), the number of nonnegative entries cannot increase if we replace $\alpha^{\star}$ by $\alpha$, which leads to the upper bound $a_i/2$.

For $a_h$ odd, consider an application of the algorithm to the graph $G^{\star \star}=P_{2a_h+1}$ (rooted at the central vertex $w$) with $\alpha^{\star \star}=2-2/(2a_h+1)$. The auxiliary result tells us that exactly $a_h$ entries are nonnegative in the end. We also know that, by symmetry, the number of nonnegative entries on each component of $G^{\star \star}-w$ must be the same, and hence is equal to $\lfloor a_h/2 \rfloor$. This implies that the number of nonnegative entries on $P_{a_h}$ when the algorithm is applied to $G$ (rooted at $u$) with $\alpha^{\star \star}$ is $\lfloor a_h/2 \rfloor$. We reach the desired conclusion by using the hypothesis $a_i<n/2$, which implies that $\alpha=2-2/n \geq  \alpha^{\star \star}$. Therefore the number of nonnegative entries on $P_{a_h}$ when the algorithm is applied to $G$ with $\alpha=2-2/n$ is bounded above by $\lfloor a_h/2 \rfloor$, as required.
\end{proof}

Furthermore, the following useful trigonometric identity has been applied in our proof of Theorem~\ref{teoenlapl}.  Although it can also be proved with trigonometric arguments, we provide a short proof which relies on spectral graph theory.

\begin{Lema}\label{lemasumcos} For any positive integer $k$, we have $$\sum_{j=1}^{k} \cos{\frac{2j\pi}{2k+1}} = -\frac{1}{2}.$$
\end{Lema}

\begin{proof} Let $G$ be a graph in $\mathcal{S}_{n,k}^{\ast}$ where the root $u$ is incident to two copies of $P_{k}$. It is well known that the sum of the Laplacian eigenvalues of a graph is twice the number of edges. Since $\mathcal{E}_{e_1}$ adds a single edge to $G$, the difference between the sum of the Laplacian eigenvalues of $G'=\mathcal{E}_{e_1}(G)$ and the sum of the Laplacian eigenvalues of $G$ is 2. This leads to
\begin{eqnarray*}
2&=&\sum_{i=1}^n \mu_i^{G'} - \sum_{i=1}^n \mu_i^{G}\\
&=&  \sum_{\mu_i^{G'} \in F_k} \mu_i^{G'} + \sum_{\mu_i^{G'}\in Lspect(G')\setminus F_k} \mu_i^{G'} - \sum_{\mu_i^G \in D_k}  \mu_i^{G}- \sum_{\mu_i^G\in Lspect(G)\setminus D_k} \mu_i^{G} \\
&=& \sum_{i=1}^{k}\left(2-2\cos{\frac{2i\pi}{2k+1}}\right) - \sum_{i=1}^{k}\left(2+2\cos{\frac{2i\pi}{2k+1}}\right) + 0\\
&=&  - 4\sum_{i=1}^k \cos{\frac{2i\pi}{2k+1}}
\end{eqnarray*}
and the result follows. Here we used Proposition \ref{propdifespectSk} which ensures that \\$Lspect(G)\setminus D_k=Lspect(G')\setminus F_k$.
\end{proof}

\bibliographystyle{acm}

\vspace{0.5cm}

\end{document}